\documentclass[11pt]{amsart}
\usepackage{pstricks,pst-plot}
\usepackage{tikz}
\usetikzlibrary{patterns}

\definecolor{Black}{cmyk}{0,0,0,1}
\definecolor{OrangeRed}{cmyk}{0,0.6,1,0} 
\definecolor{DarkBlue}{cmyk}{1,1,0,0.20}
\definecolor{myblue}{rgb}{0.66,0.78,1.00}
\definecolor{Violet}{cmyk}{0.79,0.88,0,0}
\definecolor{Lavender}{cmyk}{0,0.48,0,0}

\parskip=\smallskipamount

\newtheorem{theorem}{Theorem}[section]
\newtheorem{lemma}[theorem]{Lemma}
\newtheorem{corollary}[theorem]{Corollary}
\newtheorem{proposition}[theorem]{Proposition}

\theoremstyle{definition}
\newtheorem{definition}[theorem]{Definition}
\newtheorem{example}[theorem]{Example}

\newtheorem{remark}[theorem]{Remark}
\newtheorem{question}[theorem]{Question}

\newcommand{\C}{\mathbb{C}}

\newcommand{\R}{\mathbb{R}}

\newcommand\Subset{\subset\subset}

\newcommand{\bea}{\begin{eqnarray*}}
\newcommand{\eea}{\end{eqnarray*}}

\numberwithin{equation}{section}

\begin{document}

\title[Weighted-$L^2$ polynomial approximation ]{ Weighted-$L^2$ polynomial approximation in $\mathbb{C}$}

\author{S\'everine Biard}
\author{John Erik Forn\ae ss}
\author{Jujie Wu}

\address{ {
$^{*}$ Corresponding author  \\
Jujie Wu} \
\\{\it E-mail address:}  jujie.wu@ntnu.no
\\{ School of Mathematics and statistics, Henan University}\\{ Jinming Campus of Henan University, Jinming District, City of Kaifeng, Henan Province. P. R. China, 475001}   \\ \ \ \ \  \ \ \ \ \ \ \ \ \ \ \ \ \ \ \ \ \ \ \ \ \ \ \ \ \ \ \  \ \ \ \ \ \ \  \ \ \ \ \ \ \  \ \ \ \ \ \ \   \ \ \    \ \ \ \ \ \ \ \ \ \  \ \ \ \ \ \ \ \ \ \  \ \ \ \ \ \ \ \ \ \ \ \ \ \ \ \ \ \ \ \ \ \ \ \ \ \ \ \ \ \ \ \  \ \ \ \ \ \ \ \ \ \ \ \ \ \ \ \ \ \ \  \ \ \ \ \ \ \ \ \ \ \ \ \ \ \ \ \ \ \ \ \ \ \ \ \ \ \ \ \ 
 \  \ \ \  \ \ \ \ \ \ \ \  \ \ \ \ \ \ \ \ \ \ \ \ \    \\{Department of Mathematical Sciences, NTNU}\\{ Sentralbygg 2, Alfred Getz vei 1, 7034 Trondheim, Norway} }

\address{{S\'everine Biard}
\\{\it E-mail address:} biard@hi.is
\\{Science Institute, University of Iceland}\\{Dunhagi 3, IS-107 Reykjavik, Iceland}}

\address{{John Erik Forn\ae ss}
\\{\it E-mail address:} john.fornass@ntnu.no
\\{ Department of Mathematical Sciences, NTNU}\\{ Sentralbygg 2, Alfred Getz vei 1, 7034 Trondheim, Norway}}

\date{}
\maketitle

\bigskip

\begin{abstract}
We study the density of polynomials in $H^2(\Omega,e^{-\varphi})$, the space of square integrable holomorphic functions in a bounded domain $\Omega$ in $\mathbb{C}$, where $\varphi$ is a subharmonic function. In particular, we prove that the density holds in Carath\'{e}odory domains for any subharmonic function $\varphi$ in a neighborhood of $\overline{\Omega}$. In non-Carath\'{e}odory domains, we prove that the density depends on the weight function, giving examples.

\bigskip


\smallskip


\end{abstract}

\tableofcontents

\section{Introduction}
Let $\Omega$ be a domain in $\mathbb{C}$.
We denote by $L^2(\Omega, e^{-\varphi})$ the space of measurable functions $f$ such that 
$$\|f\|_{\Omega,\varphi}^2 : = \int_{\Omega } |f|^2 e^{-\varphi}d\lambda < + \infty,$$
where $\varphi$ is  a measurable function on $\Omega$, and $d\lambda$ is the Lebesgue measure. Let $H^2(\Omega, e^{-\varphi})$ (resp. $H^2(\overline{\Omega}, e^{-\varphi})$\,) be the space of  holomorphic functions on a domain $\Omega$ (resp. holomorphic functions on a neighborhood of $\overline{\Omega}$), which are in $L^2(\Omega, e^{-\varphi})$, that is
$$
H^2(\Omega,  e^{-\varphi}):=\mathcal{O}(\Omega)\cap L^2(\Omega, e^{-\varphi}).
$$

Recall that a Carath\'{e}odory domain $\Omega$ is a simply-connected bounded planar domain whose boundary $\partial{\Omega}$ is also the boundary of an unbounded domain. An unbounded domain $\Omega$ is said to be Carath\'{e}odory if its image under the map $z\mapsto (z-z^0)^{-1}$ ($z^0$ being a fixed point in ${\mathbb C}\backslash \overline{\Omega}$) is Carath\'{e}odory. Every Jordan domain is a Carath\'{e}odory domain. The domains, for example, of a snake winding infinitely often around the outside of a circle and approaching this circle (``outer snake'') are Carath\'{e}odory, but not a snake winding infinitely often inside a circle and approaching it from the inside (``inner snake"). For more relavant references about Carath\'{e}odory domain, please see \cite{Gaier80} on page 17.  

Questions of completeness for polynomials were first studied by T. Carleman \cite{Carleman1923} in 1923 who proved that if $\Omega$ is a Jordan domain and $ \varphi \equiv 0$, then every $L^2$ holomorphic function on $\Omega$ can be approximated  by polynomials in $L^2(\Omega, 1)$, and this result was later extended by Farrell \cite{Farrell34} and Markushevitch \cite{Mark34} independently to Carath\'{e}odory domains. For more general non-Carath\'{e}odory domains  it is well known that this property need not hold. In \cite{Hedberg65}, Hedberg proved that if $\Omega$ is a Carath\'{e}odory domain, the weight $e^{-\varphi}$ is continuous and satisfies some 
conditions then polynomials are dense in $H^2(\Omega, e^{-\varphi})$. For non-Carath\'{e}odory domains, the weighted approximation is usually considered when the weight $e^{-\varphi}$ is essentially bounded and satisfies additional conditions (see \cite{Brennan77}). Based on H\"ormander's $L^2$-estimates for the $\bar{\partial}-$operator, Taylor \cite{B.A.Taylor1971} proved that polynomials are dense in $H^2(\mathbb{C}^n, e^{-\varphi} )$ when $\varphi$ is convex, which allows the weight to have singularity and can be seen as a major breakthrough for weighted $L^2$ approximation (see also \cite{D.Wohlgelernter75}).  Sibony \cite{N.Sibony76} generalized Taylor's result and obtained that if $\varphi$ is plurisubharmonic (psh) on $\C^n$ and complex homogeneous of order $\rho >0$, i.e,  $\varphi(uz)=\vert u\vert^\rho\varphi(z)$ for $u\in\mathbb{C}$, $z\in\mathbb{C}^n$ then polynomials are dense in $H^2(\mathbb{C}^n, e^{-\varphi} )$ (see also \cite{FFW}, section 8). It is well known that each convex function is psh, but the converse is not true. Thus it is natural to ask 
\begin{question} \label{question}
Can we generalize Taylor's result to any psh function or can we find some non-convex psh function $\varphi$ on $\Omega \subset \mathbb{C}^n$ satisfying the property that $H^2(\Omega, e^{-\varphi} )$ contains all the polynomials but polynomials are not dense in it?
\end{question} 
Our first result is
\begin{proposition} \label{th:holomorphic approximation}
Let $\Omega=\bigcap^{N}_{\nu=0} G_\nu$ be a bounded domain in ${\mathbb C}$ where $G_0$ is a bounded Carath\'{e}odory domain and $G_\nu,\, 1 \leq \nu \leq N$, are unbounded Carath\'{e}odory domains. If $\varphi$ is a subharmonic function on $\overline{\Omega}$, i.e. in a neighborhood $V$ of $\overline{\Omega}$, then $H^2(\overline{\Omega}, e^{-\varphi})$ is dense in $H^2(\Omega, e^{-\varphi})$.
\end{proposition}

Our proof depends heavily on the Donnelly-Fefferman $L^2$-estimate for the $\bar{\partial}$-operator. In contrast with known results on weighted $L^2$ approximation of holomorphic functions, we allow singularities of the weight function, which makes the result useful.  
An application of Proposition \ref{th:holomorphic approximation} is the following

\begin{theorem} \label{co:polynomialapp}
Let $\Omega$ be a bounded Carath\'{e}odory domain and $\varphi$ a subharmonic function on $\overline{\Omega}$. 
then polynomials are dense in  $H^2(\Omega,e^{-\varphi})$.
\end{theorem}
Especially we will have the following 

\begin{corollary} \label{coro:Jordandense}
Let $\Omega$ be a bounded Jordan domain and let $\varphi$ be as in Theorem \ref{co:polynomialapp}. Then polynomials are  dense in $H^2(\Omega,e^{-\varphi})$.
\end{corollary}

\begin{remark}\label{rmk:Taylor}
Under the assumptions of Corollary \ref{coro:Jordandense}, let $f\in H^2(\Omega,e^{-\varphi})$. Then $f$ can be approximated by polynomials in $H^2(\Omega, e^{-\varphi})$ such that the Taylor series of the polynomials around a given point $p\in \Omega$ agrees with the one for $f$ to any given order. 
\end{remark}

It's not the case that polynomials are dense for general psh weight functions so that the corresponding Hilbert spaces contain the polynomials.
An example is provided by the following

\begin{theorem}\label{counterexinC}
Let $\varphi(z)=  |\Im m(z)| + |z|^p$, with $0 < p < 1$. Then the holomorphic polynomials are in $H^2(\C, e^{-\varphi})$, but they are not dense in $H^2(\mathbb{C}, e^{-\varphi})$.
\end{theorem}

\indent A general moon-shaped domain is a bounded domain whose boundary consists of two Jordan curves having exactly one point in common. This point is called the multiple boundary point. The moon-shaped domain is an example of a non-Carath\'eodory Runge domain in $\mathbb{C}$. Keldych \cite{Keldych1939} was the first to study the $L^2$ polynomials approximation property of moon-shaped domains without weight. Here we generalize his result with singular weight as in the following.

\begin{theorem} \label{moonshapeddomain}
Let $\Omega$ be a moon-shaped domain with the origin inside the inner Jordan curve of the boundary $\partial{\Omega}$. Let $\varphi$ be a subharmonic function on $\overline{\Omega}$.
Then polynomials are dense in $H^2(\Omega, e^{-\varphi})$  if and only if the function $\frac{1}{\sqrt{z}}$ can be approximated arbitrarily well by polynomials in $L^2(\Omega, e^{-\varphi})$.
\end{theorem}

We give two concrete examples of moon-shaped domains where density holds. The first example is based on Keldych \cite{Keldych1939} and is an application of Theorem \ref{moonshapeddomain}.\\

\begin{definition}
Let $\phi$ be a subharmonic function on $\mathbb C.$ Let $\mu$ denote the Laplacian of $ \frac {1}{2\pi}\phi.$ We say that $\phi$ satisfies \textsl{condition (A)} if the mass of $\mu$ on the closed unit disc is strictly less than 2.
\end{definition}

\begin{example} \label{moonshapeddomainarbitrarily}
There exists a moon-shaped domain with the unit circle being the outer Jordan curve, such that for any subharmonic function $\varphi$ on $\mathbb{C}$ satisfying condition (A),  polynomials are in $H^2(\Omega,e^{-\varphi})$ and dense in it.
\end{example}

\begin{example} \label{moonshapeddomainarbitary}
There exist a moon-shaped domain bounded by two circles and a subharmonic function $\varphi$ on $\Omega$ so that polynomials are in $H^2(\Omega,e^{-\varphi})$ and are dense in it.
\end{example}


\begin{theorem} \label{moonshapeddomainnot}
Let $\Omega$ be a moon-shaped domain bounded by two circles and let $\varphi$ be a subharmonic function on $\Omega$ which is uniformly bounded above. 
Then the set of polynomials which is in $H^2(\Omega, e^{-\varphi})$ is never a dense subset of $H^2(\Omega,e^{-\varphi})$.
\end{theorem}

This paper is set-up as follows. In Section 2, we prove Proposition \ref{th:holomorphic approximation} and Theorem \ref{co:polynomialapp}, exploiting the property of Carath\'{e}odory domains in order to be able to exhaust them from outside by Jordan domains that are conformally equivalent to the unit disc. Then, we apply Donnelly-Fefferman's $L^2$-estimates on each of those to obtain the weighted $L^2$ approximation. In Section 3, we give a counterexample on $\mathbb{C}$ where we exhibit a subharmonic function $\varphi$ for which the polynomials are in $H^2(\mathbb{C}, e^{-\varphi})$ but they are not dense in it (Theorem \ref{counterexinC}). 
In Section 4, we prove Theorem \ref{moonshapeddomain} and give an example of a moon-shaped domain where the density is proved by approximating $\frac{1}{\sqrt{z}}$ (Example \ref{moonshapeddomainarbitrarily}). In Section 5, we present Example \ref{moonshapeddomainarbitary} and we finally prove, in Section 6,  Theorem \ref{moonshapeddomainnot}.

\section{Proof of Theorem \ref{co:polynomialapp}}

We observe 
that it suffices to prove Proposition \ref{th:holomorphic approximation}
and Theorem \ref{co:polynomialapp} when the subharmonic function $\varphi$ is globally defined.
To see this, let $\phi$ be a subharmonic function defined on a bounded open set $V$ containing $\overline{\Omega}$
and choose an open set $U, \overline{\Omega}\subset U\Subset V.$ Then $\mu:= \Delta(\phi)_{\mid_U}$ is a positive measure with bounded mass on $\mathbb C$. Hence there is a globally defined subharmonic function 
$\varphi$ such that $\Delta(\varphi)=\mu.$ But then $\phi=\varphi+h$ on $U$ for some harmonic function $h$.
Since $h$ is uniformly bounded on $\Omega$, it follows that
the Hilbert spaces $H^2(\Omega,e^{-\phi})$ and $H^2(\Omega, e^{-\varphi})$ are the same and the norms are equivalent.

We use, in the next Lemma, the following classical result from one complex variable:

\begin{theorem} [cf.\,\cite{Tsuji59}, p.\,382\,]\label{th:caratheodory}
Let $\{\Omega_n\}_{n=1}^\infty$ be a sequence of uniformly bounded simply connected domains in ${\mathbb C}$ and $\Omega$ a bounded simply connected domain,  all containing the origin, so that the Hausdorff distance between $\partial \Omega_n$ and $\partial \Omega$ tends to zero as $n\rightarrow \infty$. If we map ${\mathbb D}$ conformally onto $\Omega_n$ by $w= f_n(z)$, $f_n(0) = 0$, $f'_n(0) > 0$, then $f_n$ converges locally uniformly to $f\in {\mathcal O}({\mathbb D})$ such that $w=f(z)$ maps ${\mathbb D}$ conformally onto $\Omega$.
\end{theorem}

%

For a planar domain $\Omega$, let $SH^-(\Omega)$ denote the set of negative subharmonic functions on $\Omega$. Our key observation to prove Proposition \ref{th:holomorphic approximation} is the following:

\begin{lemma} \label{le:caratheodory}
 Let $\Omega$ be a bounded Carath\'{e}odory domain. Then there exists a sequence of bounded simply-connected domains $\Omega_n\supset \overline{\Omega}$, a sequence of positive numbers $\varepsilon_n\rightarrow 0$ $(n\rightarrow \infty)$, and a sequence of continuous functions $\rho_n\in SH^{-}(\Omega_n)$ such that

$(1)$ $\Omega_{n,-\varepsilon_n}:=\{z\in \Omega_n:\rho_n(z)<-\varepsilon_n\}\subset \Omega$,

$(2)$ the volume of $\Omega \setminus \Omega_{n,-2\varepsilon_n}$ tends to 0 as $n\rightarrow\infty$.

\end{lemma}

Before proving the Lemma we recall the following $L^2$-estimates for the $\bar\partial$-operator which will be used here.

\begin{proposition}[Donnelly-Fefferman, \cite{DF83}]\label{th:D-Fth}
Let $\Omega \subset \mathbb{C}^n$ be a pseudoconvex domain and $\varphi \in psh(\Omega)$. Suppose that $\psi $ is a $C^2$ strictly psh function which satisfies
\begin{eqnarray}\label{eq:assumption}
i\partial\overline{\partial} \psi \geq i\partial \psi \wedge \overline{\partial} \psi.
\end{eqnarray}
Then for each $\overline{\partial}$-closed $(0,1)$-form $v$, there exists a solution $u$ to $\overline{\partial} u = v$ satisfying
\begin{eqnarray}\label{eq:estimate1}
\int _{\Omega} |u|^2e^{-\varphi}d\lambda \leq C\int _{\Omega} |v|_{i\partial \overline{\partial} \psi}^2e^{-\varphi}d\lambda,
\end{eqnarray}
where $C>0$ is an absolute constant, provided that the right hand side of (\ref{eq:estimate1}) is finite.
\end{proposition}

The norm $\Vert \alpha\Vert_{i\partial\bar{\partial}\psi}$ for a $(0,1)$-form $\alpha$ is the smallest function $H$ that satisfies
$$i\overline{\alpha}\wedge {\alpha}\leq H (i\partial\bar\partial \psi).$$
In particular, on $\mathbb{C}$, we have
\begin{equation}
\vert \alpha\vert_{i\partial\bar\partial\psi}=\left(\dfrac{\partial^2\psi}{\partial z\partial\bar{z}}\right)^{-1}\vert\alpha\vert^2.\label{DFnorm}
\end{equation}
For more details, see for example \cite{Demailly82, Blocki14}.

 \begin{proof}[ Proof of Lemma \ref{le:caratheodory}]
Since $\Omega$ is a Carath\'{e}odory domain, there exists a sequence $\{\Omega_n\}$ of bounded 
simply-connected domains such that $\overline{\Omega}\subset \Omega_n$ and ${\overline{\Omega}_{n+1}}\subset \Omega_n$ and the Hausdorff distance between $\partial \Omega_n$ and $\partial\Omega$ tends to zero as $n\rightarrow \infty$  (see for example \cite{Gaier80}, p.17). 
Without loss of generality, we may assume that $0\in\Omega$. By virtue of Riemann's mapping theorem, there are  conformal mappings
$$
w=f_n(z),\; f_n(0)=0,\; f_n'(0)>0
$$
which maps $\Omega_n$ onto ${\mathbb D}$,
and
$$
w=f(z),\; f(0)=0,\; f'(0)>0,
$$
which maps $\Omega$ onto ${\mathbb D}$.
Set
$$
\rho(z)=|f(z)|-1,\ \ \  \rho_n(z)=|f_n(z)|-1.
$$
Clearly,  $\rho$ (resp. $\rho_n$) is a negative continuous subharmonic function on $\Omega$ (resp. $\Omega_n$).
Let
$$
\varepsilon_n:=\max\left\{1-|f_n(z)|: z\in \overline{\Omega}_n\backslash \Omega\right\}.
$$
By Theorem 2.1, the sequence of Riemann mappings $f^{-1}_n:\mathbb D\rightarrow \Omega_n$ converges u.c.c. to
the Riemann mapping
$f^{-1}:\mathbb D\rightarrow \Omega.$ 
Suppose $z\in \Omega_n$ and $\rho_n(z)<-\varepsilon_n.$ Then
$|f_n(z)|-1<-\varepsilon_n$ and hence $1-|f_n(z)|>\varepsilon_n.$ This implies that $z\in \Omega$
which proves (1).
Next we prove that $\varepsilon_n\rightarrow 0.$ 
Let $0<a<b<c<1$. 
By the open mapping theorem, there exists $n_0\in\mathbb{N}$ so that 
$$
f^{-1}(|z|<a) \subset f^{-1}_n(|z|<b)\subset f^{-1}(|z|<c)\; ,  \ \  \forall \;n\geq n_0.
$$

Hence $f^{-1}_n(|z|<b)\subset \Omega$ so $f_n(\Omega)\supset (|z|<b).$
Therefore if $z\in \Omega_n\setminus \Omega,$ then $|f_n(z)|\geq b,$ so
$1-|f_n(z)|\leq 1-b.$ This shows that $\varepsilon_n\leq 1-b$ if $n\geq n_0.$ It follows that
$\varepsilon_n\rightarrow 0$ as $b$ tends to $1$.

Finally we show (2).

Let $\delta>0.$ Then if $0<a<1$ is chosen large enough, then the area $|f^{-1}(|z|>a)|<\delta.$
Choose $0<a<b<c<1$ as above with $1-b<(1-a)/2.$ Then for all large enough $n$, $2\varepsilon_n\leq 2(1-b)< 1-a.$ 
Suppose that $z\in\Omega\setminus \Omega_{n,-2\varepsilon_n}$.
Then $\rho_n(z)\geq -2\varepsilon_n$ so $|f_n(z)|\geq 1-2\varepsilon_n>a$ which implies that
(for large $n$) $|f(z)|>a.$ Hence $|\Omega \setminus \Omega_{n,-2\varepsilon_n}|<\delta.$

\end{proof}

Now we can prove Proposition \ref{th:holomorphic approximation}.

\begin{proof}[ Proof of Proposition \ref{th:holomorphic approximation}] In view of Lemma \ref{le:caratheodory}, there exist for each $0\le \nu\le N$ a sequence of Jordan domains  $G_n^\nu\supset \overline{G}_\nu$,  a sequence of positive numbers $\varepsilon_n^\nu\rightarrow 0$ $(n\rightarrow \infty)$, and a sequence of continuous functions $\rho^\nu_n\in SH^{-}(G^\nu_n)$ such that

(i) $G^\nu_{n,-\varepsilon^\nu_n}:=\{z\in G^\nu_n: \rho_n^\nu (z) < -\varepsilon_n^\nu\} \subset G_\nu$,

(ii) ${\rm vol\,}(G_\nu\setminus G^\nu_{n,-2 \varepsilon^\nu_n})\rightarrow 0 $ as $n\rightarrow \infty$.

(In (ii) we can use the spherical metric near $\infty$.)
Set
$$
\rho_n(z)=\max_{0 \leq \nu \leq N}\{\rho_n^\nu(z)\},\ \ \ \varepsilon_n=\max_{0\le \nu \le N}\{\varepsilon_n^\nu \},\ \ \ \Omega_n= \bigcap^{N}_{\nu=0} G _n^\nu.
$$
It is easy to verify that $\overline{\Omega}\subset \Omega_n$, $\rho_n\in SH^{-}(\Omega_n)\cap C(\Omega_n)$ and

(iii)  $\Omega_{n,-\varepsilon_n}:=\{z\in \Omega_n:\rho_n(z)<-\varepsilon_n\}\subset \Omega$,

(iv) ${\rm vol\,}(\Omega\setminus\Omega_{n,-2\varepsilon_n})\rightarrow 0$ as $n\rightarrow \infty$.

We continue with the proof in a similar way as in the proof in \cite{Chenzhang2000}. 
Choose a family of negative $C^\infty$ subharmonic functions $\{\rho_{n,s}\}$ on $\Omega_{n}$ such that $\rho_{n,s} \downarrow \rho_n$ uniformly on $\overline{\Omega}_{n+1}$ as $s \downarrow 0$. Put $\psi^s_n=-\log(-\rho_{n,s})$. Clearly, we have
\begin{equation} \label{eq:dbounded}
i\partial\overline{\partial}\psi^s_n \geq i\partial \psi^s_n \wedge \bar{\partial} \psi^s_n.
\end{equation}
Now choose a cut-off function $\chi: {\mathbb R}\rightarrow [0,1]$ such that $\chi\mid_{(-\infty, -\log 3/2)} \equiv 1$ and $\chi\mid_{[0,\infty)} \equiv 0$. Set $\eta_n^s =\chi(\psi_n^s+\log\varepsilon_n)$ on $\Omega_n.$ Then we have
$$
{\rm supp\,} \eta_n^s \subset\{z\in \Omega_{n}\mid\rho_{n,s}(z)<-\varepsilon_n\}\subset \Omega_{n,-\varepsilon_n}\subset \Omega
$$
and
$$
|\overline{\partial}\eta_n^s|^2_{i\partial\overline{\partial}\psi_n^s}\leq\sup|\chi'|^2
$$
in view of (\ref{eq:dbounded}).
Here $|\cdot|_{i\partial\overline{\partial}\psi_n^s}$ stands for the point-wise norm with respect to the metric $i\partial\overline{\partial}\psi_n^s$, like in \eqref{DFnorm}. For each $f\in H^2(\Omega, e^{-\varphi})$, we define
$$
v_n^s:=f\overline{\partial}\eta_n^s.
$$
Clearly, $v_n^s$ is a well-defined $C^\infty$, $\bar{\partial}-$closed (0,1) form on $\Omega_{n}$ satisfying
\begin{eqnarray} \label{eq:solvedbar}
\int_{\Omega_{n+1}} |f|^2|\bar{\partial}\eta^s_n|^2_{i\partial\overline{\partial} \psi_n^s} e^{-\varphi} d\lambda
&  \leq & \sup |\chi'|^2 \int_{\{z\in \Omega_{n+1}:-\frac32 \varepsilon_n \leq   \rho_{n,s}(z) < -\varepsilon_n \}}  |f|^2  e^{-\varphi} d\lambda \nonumber\\
&  \leq & \sup |\chi'|^2 \int_{\Omega\backslash \Omega_{n,-2\varepsilon_n}} |f|^2 e^{-\varphi}d\lambda
\end{eqnarray}
provided $s$ sufficiently small. By Proposition \ref{th:D-Fth}, 
there exists a solution $u_n^s$ to the equation $\bar{\partial}u_n^s=v^s_n$ on $\Omega_{n+1}$ verifying, by using \eqref{eq:solvedbar},
$$
\int_{\Omega_{n+1}}|u_n^s|^2e^{-\varphi} d \lambda \leq C \sup |\chi'|^2 \int_{\Omega\backslash \Omega_{n,-2\varepsilon_n}} |f|^2 e^{-\varphi} d \lambda.
$$
Let $f_n^s =f\eta_n^s-u_n^s$. Then, $f_n^s\in {\mathcal O}(\Omega_{n+1})$ and
\begin{eqnarray*}
\int_{\Omega} |f_n^s-f|^2 e^{-\varphi}  d\lambda & \leq & C  \int_{\Omega\backslash \Omega_{n,-2\varepsilon_n}} |f|^2 e^{-\varphi} d\lambda\rightarrow 0\mbox{ as }n\rightarrow \infty \\
\end{eqnarray*}
\end{proof}

\begin{proof} [Proof of Theorem \ref{co:polynomialapp}]

 Let $\varphi$ be a subharmonic function on $\overline{\Omega}$. If for each $x \in \overline{\Omega}$, the Lelong number $\nu(\varphi)(x)$ of $\varphi$ satisfies $\nu(\varphi)(x) < 2$ then we have $1\in H^2(\Omega, e^{-\varphi})$.
 So there exists a constant $M>0$ such that $\displaystyle{\int_\Omega \mathrm{e}^{-\varphi}d\lambda<M}$. Let $f \in H^2(\Omega, e^{-\varphi})$. According to the proof of Proposition \ref{th:holomorphic approximation}, for each $\varepsilon >0$ there exists $F  \in H^2(\Omega_n, e^{-\varphi})$ satisfying 
$$
\int_\Omega |F(z)-f(z)|^2 \mathrm{e}^{-\varphi(z)}d\lambda< \frac {\varepsilon} {4}.
$$
\noindent where $\Omega_n$ is a simply connected domain containing $\overline{\Omega}.$
We apply Runge's theorem to $F$ and see that for  $\delta= \sqrt {\frac{\varepsilon}{ 4  M}}$ there exists a polynomial $P$ such that 
\begin{eqnarray*}
| F(z) - P(z) | < \delta,   \ \ \ \ \ \    z \in \overline{ \Omega }.
\end{eqnarray*}
Consequently we have 
\begin{eqnarray*}
\int_ \Omega |F(z)- P(z)|^2 e^{-\varphi} d \lambda < \delta^2 \cdot M =\frac{\varepsilon}{4}.
\end{eqnarray*}
Thus we have
\begin{eqnarray*} \label{eq:22}
\int_\Omega |f(z)-P(z)|^2 \mathrm{e}^{-\varphi(z)}d\lambda&\leq& 2 \int_\Omega |f(z)-F(z)|^2 \mathrm{e}^{-\varphi(z)}d\lambda
\\
&&+2 \int_\Omega |F(z)-P(z)|^2 \mathrm{e}^{-\varphi(z)}d\lambda
\\
& < & \varepsilon\ \
\end{eqnarray*}
If there exist finitely many points $x_0, x_1, \cdots, x_N$ with $\nu(\varphi)(x_j)\geq 2$ for each $0\leq j\leq N$, then we may choose some polynomial $Q$ so that $Q(x_j)=0$ and $\varphi = \psi + \log |Q|^2$ with $\nu(\psi)<2$ at each point $x_j$. Then for each $f\in H^2(\Omega, e^{-\varphi})$ we have 
 $$
 \int_{\Omega} \left| \frac f Q\right|^2 e^{-\psi} d \lambda < \infty. 
 $$
 Since $\psi$ is bounded above, $\int_{\Omega}\left|\frac fQ \right|^2 d \lambda < \infty$. Hence $\frac f Q $ is holomorphic on $\Omega$. Based on the above discussion, for each $\varepsilon >0$ we can find some polynomial $P$ satisfying
 $$
 \int_{\Omega} \left| \frac f Q - P \right|^2 e^{-\psi} d \lambda < \varepsilon.
$$
 That is 
 $$
 \int_{\Omega} \left| f - P Q  \right|^2 e^{-\varphi} d \lambda < \varepsilon.
$$
 
\end{proof}

\section{Proof of Theorem \ref{counterexinC} } 

In this section, we denote $x=\Re e(z)$ and $y=\Im m(z)$.

Before proving Theorem \ref{counterexinC} by contradiction, we need a couple of Lemmas:
\begin{lemma}\label{lemma}
 $ \cos \frac{z}{2} \in H^2(\mathbb{C}, e^{-\varphi})$.
\end{lemma}

\begin{proof}
Since $\cos \frac z2 =  \frac{e^{i\frac z2} + e^{-i\frac z2}}{2}$, we have
\begin{align}
\int_{\mathbb{C}} \left|\cos \frac z2\right|^2 e^{-\varphi} d \lambda & =  \int_{\mathbb{C}} \left| \frac{e^{i\frac z2} + e^{-i\frac z2}}{2}\right|^2 e^{-\varphi} d \lambda\notag  \\
& = \frac 14\int_{\mathbb{C}}\left(  \left|e^{i\frac z2} \right|^2   +   \left|e^{-i\frac z2} \right|^2  + e^{i\frac z2} \overline{e^{-i\frac z2}}   + e^{-i\frac z2} \overline{e^{i\frac z2}}  \right)e^{-\varphi} d \lambda\notag  \\
& \leq    \frac 14\int_{\mathbb{C}}\left(  e^{-y} + e^{y} + e^{ix}  + e^{-ix}  \right)e^{-\varphi} d \lambda\notag \\
& =   \frac 14 \int_{\mathbb{C}}\left( e^{-y} + e^{y} \right)e^{-|y| - |z|^p} d \lambda  +\frac12\int_{\mathbb{C}} \cos x \; e^{-|y|-|z|^p} d \lambda.\label{cos^2}
\end{align}
Remark that $(e^{-y} + e^y) e^{-|y|} = 1+ e^{-2|y|}  \leq 2$ so we obtain for the first term of the right-hand side of \eqref{cos^2}
\begin{align*}
\int_{\mathbb{C}} \left( e^{-y} + e^{y} \right)e^{-|y| - |z|^p} d \lambda  & =  \int_{\mathbb{C}} \left( e^{-2|y|} + 1 \right) e^{-|z|^p}d \lambda\notag \\
 & \leq  2 \int_{\mathbb{C}} e^{-|z|^p} d \lambda\notag \\
 & \leq   4\pi \int_0^{+ \infty} e^{-r^p} r dr < \infty, \label{e^y}
 \end{align*}
and for the second term of \eqref{cos^2},
\begin{equation*}  
 \int_{\mathbb{C}}|  \cos x | e^{-|y|-|z|^p} d \lambda  \leq \int_{\mathbb{C}}  e^{-|y|-|z|^p} d \lambda  
  \leq   2\pi \int_0^{+ \infty} r{e^{-r^p} }dr < \infty. \label{ineq:cos}
\end{equation*}
\end{proof}
We will need the following integral representation too:
\begin{lemma}[see Chapter 7 of \cite{Arlex2013}] \label{le:harmonicextension}
Let $u(t): \mathbb{R} \rightarrow  \mathbb{R}$ be a continuous function satisfying
$$
\int_{\R}\frac{u(t)}{1+|t|^2}dt < \infty.
$$
Then
$$
U(x+iy) =  \frac1 \pi \int_{-\infty} ^ {+\infty}  u(t) \frac{y}{(t-x)^2 + y^2}dt
$$
is a harmonic extension of $u$ to the upper half plane.
\end{lemma}



In particular, if $u(t) = |t|^p$, $0< p<1$, then 
\begin{equation}
U(x+ i y ) = \frac1\pi \int_{-\infty}^{+\infty} \frac{y |t|^p}{(x-t)^2 + y^2} dt, \quad y > 0,\label{U}
\end{equation}
is a harmonic extension of $|t|^p$ to the upper half plane.

\begin{proposition}\label{eq:guji1}
Let $U$ be as in \eqref{U}. Then there exists constant $C_p>\frac 1 4 $ so that on the upper half plane $y>0$,
$$
\frac14|z|^p<U(x+iy)<C_p |z|^p.
$$
\end{proposition}

\begin{proof}
The right hand side inequality is direct:

\begin{align}
U(x+iy) & =  \frac{1}{\pi} \int_{-\infty}^\infty\frac{y|t|^p}{(t-x)^2+y^2}dt\notag\\
& =  \frac{1}{\pi} \int_{-\infty}^\infty\frac{y|s+x|^p}{s^2+y^2}ds\notag\\
& =  \frac{1}{\pi} \int_{-\infty}^\infty\frac{y|y\tau+x|^p}{(y\tau)^2+y^2}yd\tau\notag\\
& =   \frac{1}{\pi} \int_{-\infty}^\infty\frac{|y\tau+x|^p}{\tau^2+1}d\tau\label{equality 1}\\
& \leq   \frac{1}{\pi} \int_{-\infty}^\infty\frac{y^p|\tau|^p}{\tau^2+1}d\tau+  \frac{1}{\pi} \int_{-\infty}^\infty\frac{|x|^p}{\tau^2+1}d\tau\notag\\
& \leq \frac 2\pi |y|^p\cdot \pi \frac{\sin (\frac p2\pi)}{ \sin(p\pi)} +|x|^p\notag\\
& =  \frac{1}{\cos (\frac p2\pi) }|y|^p + |x|^p\notag\\
& \leq   \frac{2}{\cos (\frac p2\pi) } |z|^p\notag\\
&: =  C_p |z|^p.
\end{align}

To prove the left-hand side inequality, we prove the following inequalities: (i) $U(x+iy)\geq \frac12 |x|^p$ and (ii) $U(x+iy)\geq \frac12|y|^p.$

We prove (i) as follows: from \eqref{equality 1}
$\tau$ inherits the sign of $x$, then we only have to study $x\geq0$.
We get
\bea
U(x+iy) & =  & \dfrac{1}{\pi} \int_{-\infty}^\infty\frac{|y\tau+x|^p}{\tau^2+1}d\tau\\
& \geq & \dfrac{1}{\pi} \int_{0}^\infty\frac{|y\tau+x|^p}{\tau^2+1}d\tau\\
& \geq & \dfrac{1}{\pi} \int_{0}^\infty\frac{|x|^p}{\tau^2+1}d\tau\\
& \geq & \dfrac12|x|^p.\\
\eea

For (ii), we also start from \eqref{equality 1} and by a similar argument, we may assume that $x\geq 0$. Then,

\bea
U(x+iy) & =  & \dfrac{1}{\pi} \int_{-\infty}^\infty\frac{|y\tau+x|^p}{\tau^2+1}d\tau\\
& \geq & \dfrac{1}{\pi} \int_{0}^\infty\frac{|y\tau+x|^p}{\tau^2+1}d\tau\\
& \geq & \dfrac{1}{\pi} \int_{0}^\infty\frac{|y|^p\tau^p}{\tau^2+1}d\tau\\
& = &  \dfrac{|y|^p}{\pi} \int_{0}^\infty\frac{\tau^p}{\tau^2+1}d\tau\\
& = & \dfrac12\frac{|y|^p}{\cos (\frac p2\pi) }\\
& \geq &  \frac 12 |y|^p.
\eea
Finally, combining (i) and (ii) and by concavity for $0<p<1$, we get 
$$
U(x+iy) \geq \frac14 (|x|^p + |y|^p) \geq \frac14 (|x|+|y|)^p \geq \frac14|z|^p.
$$
\end{proof}

\begin{proof} [Proof of Theorem \ref{counterexinC}]

First, remark that the holomorphic polynomials are in $L^2(\mathbb{C}, e^{-\varphi})$ thanks to the exponential rate $e^{-\vert z\vert^p}$.\\

We prove Theorem \ref{counterexinC} by contradiction.  
Assume now that holomorphic polynomials are dense in $H^2(\mathbb{C}, e^{-\varphi})$. Since $\cos \frac z2 \in H^2(\mathbb{C}, e^{-\varphi})$ by Lemma \ref{lemma}, for all $\varepsilon >0$, there exist a sequence of polynomials $(P_n)$ and $N\in\mathbb{N}$ such that for all $n\geq N$,
\begin{equation}\label{eq:approximation1}
\int _{\mathbb{C}} \left|P_n(z)-\cos \frac z2 \right|^2 e^{-\varphi}d \lambda  < \varepsilon.  
\end{equation}
Note that for $n$ sufficiently large,
\begin{align}
\|P_n(z)\|_{\mathbb{C},\varphi} & =  \left\Vert P_n(z) -\cos \frac z2 + \cos \frac z2\right \Vert_ {\mathbb{C},\varphi}\notag\\
& \leq   \left\Vert P_n(z) -\cos \frac z2\right\Vert_ {\mathbb{C},\varphi} + \left\Vert\cos \frac z2 \right\Vert_ {\mathbb{C},\varphi}\notag\\
& \leq  1+ \left\Vert\cos \frac z2 \right\Vert_{\mathbb{C},\varphi}.\label{Pn}
\end{align}
We deduce from \eqref{Pn} that there exists $M>1$ such that $\|P_n(z)\|_{\mathbb{C},\varphi} \leq M$ for all $n$.

Since $P_n$ is analytic, we have
$$
P_n(z) = \frac{1}{\pi} \int_ {|\zeta| \leq 1} P_n(z + \zeta) d\lambda_\zeta.
$$
So
\begin{align}
|P_n(z)| &\leq \dfrac{1}{\pi} \int_{|\zeta| \leq 1} |P_n(z+ \zeta)| e^{\frac {\varphi(z+\zeta)}{2}} e^{-\frac {\varphi(z+\zeta)} {2}}d\lambda_\zeta\notag \\
& \leq  \dfrac{1}{\pi} \sup\limits_{|\zeta| \leq 1}e^{\frac{\varphi(z+\zeta)}{2}}  \int_ {|\zeta| \leq 1}|P_n(z+ \zeta)| e^{-\frac {\varphi(z+\zeta)}{2}}d\lambda_\zeta.\label{Pn2}
\end{align}
By Cauchy-Schwarz's inequality, we get from \eqref{Pn2} and for all $n$, 
\begin{align}
|P_n(z)| & \leq \dfrac{1}{\pi} \sup\limits_{|\zeta| \leq 1}e^{\frac{\varphi(z+\zeta)}{2}}   \left (  \int_ {|\zeta| \leq 1}|P_n(z+ \zeta)|^2 e^{-{\varphi(z+\zeta)}}d\lambda_\zeta \right) ^\frac12  \left(  \int_{|\zeta| \leq 1}d\lambda_\zeta \right)^\frac12 \notag\\
&\leq \frac{1}{\sqrt {\pi}} \|P_n(z) \|_{\mathbb{C},\varphi} \sup\limits_{|\zeta| \leq 1}e^{\frac{\varphi(z+\zeta)}{2}} \leq \frac{M}{\sqrt {\pi}} \sup\limits_{|\zeta| \leq 1}e^{\frac{\varphi(z+\zeta)}{2}}.\label{Pn3}
\end{align}
Note that for $0<p<1$,
\bea
\sup\limits_{|\zeta| \leq 1}  \varphi(z+\zeta) & = & \sup\limits_{|\zeta| \leq 1} \left(| \Im m (z+ \zeta) | + |z+\zeta| ^p \right) \\
&\leq &  |\Im m(z)| + 1+ (|z| + 1)^p \\
& \leq & |\Im m(z)| + 1+1 + |z|^p\\
& = &  |y| +2 + |z|^p.\\
\eea
Plugging it in \eqref{Pn3}, we get for all $n$
\bea
| P_n( x + iy ) |  \leq \frac{M}{\sqrt {\pi}} e^{\frac12 (2+ |y| + |x+iy|^p)}.
\eea
Then it follows that on the real axis $(y = 0)$,
\bea
| P_n( x ) |  \leq \dfrac{M}{\sqrt {\pi}} e^{\frac12 (2+  |x|^p)}  \leq \dfrac{Me}{\sqrt {\pi}} e^{\frac 12|x|^p},
\eea
so
\begin{equation} \label {eq:1}
\log |P_n(x)| \leq  \log M + 1 - \log \sqrt{\pi} + \frac 12 |x|^p.
\end{equation}
By Lemma \ref{le:harmonicextension} applied to $\tilde{u}(x)=\log M + 1 - \log \sqrt{\pi} + \frac 12 |x|^p$, 
\bea
\widetilde{U}(x+ i y ) = \frac1\pi \int_{-\infty}^{+\infty} \frac{y( \log M +1 - \log \sqrt{\pi} +\frac12 |t|^p)}{(x-t)^2 + y^2} dt,  \ \ \ y > 0,
\eea
is a harmonic extension of $\tilde{u}$ to the upper half plane. And from Proposition \ref{eq:guji1}, we obtain
$$
 \log M +1 - \log \sqrt{\pi} +\frac18 |z|^p  \leq \widetilde{U}(z) \leq  \log M +1 - \log \sqrt{\pi}  + C_p|z|^p.
 $$
From the fact that for any positive constant $C$, $C\log |z|$ is much smaller
than $|z|^p$ for large enough $|z|$, we deduce that
for large enough $|z|,$ $\log |P_n|(z)\leq \widetilde{U}(z)$. Because $\widetilde{U}$ is harmonic and $\log|P_n|$ is subharmonic, we obtain

\begin{equation} \label{eq:hamonicextension1}
\log |P_n(x+iy)| \leq  \widetilde{U}(x+iy) \leq C_1 +C_p |z|^p\quad\text{on $\lbrace{z\in\mathbb{C}\mid y > 0\rbrace}$},
\end{equation}
where $C_1 = \log M +1 - \log \sqrt{\pi}.$
From \eqref{eq:hamonicextension1}, we get 
$$
\log |P_n(z)| \leq C_1 +C_p |z|^p,
$$
so
$$|P_n(z)|\leq e^{C_1} e^{C_p|z|^p}.$$
But remark that
$$
\lim_{|z| \rightarrow +\infty}\frac{e^{|z|/4}-1}{ e^{C_1} e^{C_p|z|^p}} =+\infty
$$
so there exists a positive constant $Y>1$ such that for all $\vert z\vert>Y$,
\begin{equation}
e^{|z|/4}-1> 4 e^{C_1} e^{C_p|z|^p}.\label{e^z4}
\end{equation}
Note also that if $x^2 <3y^2$ then $4y^2 =y^2+3y^2> x^2+y^2 =|z|^2$ so if in addition $y>0$, we get $y > \frac12 |z|$. Hence, on $\lbrace{z\in\mathbb{C}\mid x^2 <3y^2,\; y>0}\rbrace$, we have 
\begin{align}
\left| \cos \frac z2 \right| &= \left| \frac 12  e^{ix/2-y/2} +\frac12 e^{ix/2+y/2}\right|\notag \\
&\geq  \left| \frac12 e^{ix/2+y/2} \right| - \left| \frac12 e^{ix/2-y/2} \right|\notag\\
&=  \frac12 (e^{y/2}  - e^{-y/2})\notag\\
&\geq \frac12 (e^{|z|/4}-1).\label{cos}
\end{align}
Combining \eqref{cos} with \eqref{e^z4}, we obtain on $W:=\{ z\in\mathbb{C}\mid x^2 < 3y^2,\; y>0,\; |z|> Y \}=\lbrace{(r,\theta)\mid r>Y,\quad \frac{\pi}{6}<\theta<\frac{5\pi}{6}}\rbrace$,
\begin{equation}
\left|   \cos \frac z2 - P_n(z)\right| \geq \left|   \cos \frac z2 \right| - \left| P_n(z)\right| \geq e^{C_1} e^{C_p|z|^p}  \geq e^{C_1} e^{C_pY^p}.\label{eq:2}
\end{equation}
Hence, from \eqref{eq:2}, we have on $W$,
\bea
\int _{\mathbb{C}} \left|   \cos \frac z2 - P_n(z)\right| ^2 e^{-\phi(z)}d\lambda &\geq& \int _W\left|   \cos \frac z2 - P_n(z)\right| ^2 e^{-\varphi(z)}d\lambda\\
&\geq & e^{2C_1} e^{2C_pY^p} \int_W e^{-\varphi(z)}d\lambda \\
& = & e^{2C_1+2C_pY^p} \int_{\pi/6}^{5\pi/6} d\theta \int_Y^{+\infty} re^{-|r\sin\theta|-|r|^p}dr\\
& \geq  & \frac{2\pi}{3} e^{2C_1+2C_pY^p}\int_Y^{+\infty} re^{-|r|-|r|^p}dr \\
& \geq  & \frac{2\pi}{3} e^{2C_1+2C_pY^p} \cdot  \frac12 e^{-2Y}\\
&=& \frac{\pi}{3} e^{2C_1+2C_pY^p-2Y} .\\
\eea
This is a contradiction with the formula $(\ref{eq:approximation1})$ when $n$ is large enough.
\end{proof}

\section{Proof of Theorem \ref{moonshapeddomain}}
Recall the following classical fact:
\begin{lemma}
Let $\Omega$ be an open set in $\mathbb{C}$, $h$ a  holomorphic on $\Omega$ and $\varphi$ a subharmonic function on an open set $V\supset h(\Omega)$. Then  $\varphi \circ h(z) = \varphi(h(z))$ is also subharmonic on $\Omega$.
\end{lemma}
\begin{proof} [Proof of Theorem \ref{moonshapeddomain}]
Since $\frac{1 }{\sqrt{z}} \in H^2(\Omega, e^{-\varphi})$, the condition is obviously necessary.\\

 We then need to prove the sufficiency. The mapping $w = \sqrt{z}$ transforms $\Omega$ into a Jordan domain $\Omega'$ and for each $f\in H^2(\Omega, e^{-\varphi})$, 
\begin{equation*}
\int_{\Omega} |f(z)|^2 e^{-\varphi(z)}d\lambda_z  = 4 \int_{\Omega'} |f(w^2)w|^2 e^{-\varphi(w^2)}d \lambda_w < \infty.
\end{equation*}
Put $h(w) = w^2$, we know that $wf(w^2) \in H^2(\Omega', e^{-\varphi \circ h(w) })$, where $\varphi \circ h(w) = \varphi(h(w))$ is subharmonic on the closure of the bounded Jordan domain $\Omega'$. In particular, $\varphi \circ h(w) = \varphi(h(w))$ is subharmonic in a neighborhood of $\overline{\Omega'}$. 
 According to Corollary \ref{coro:Jordandense}, for each $\varepsilon >0$ there exists a polynomial $P(w)$ such that 

\begin{equation}
4 \int_{\Omega'} | wf(w^2) - P(w)|^2 e^{-\varphi \circ h(w)} d \lambda_w < \varepsilon. \nonumber
\end{equation}
Therefore 
\begin{equation}\label{polydense}
\small{ \int_{\Omega} | \sqrt {z} f(z) - P(\sqrt {z})|^2 e^{-\varphi (z)}\cdot   \frac{1}{\vert z\vert }d \lambda_z   =   \int_{\Omega}\left  | f(z) - \frac{P(\sqrt {z})}{\sqrt{z}} \right |^2 e^{-\varphi (z)} d \lambda_z < \varepsilon.  }
\end{equation}
Separate the polynomial $P$ into even and odd parts:
\begin{equation*}
P(\sqrt{z}) = P_1(z ) + \sqrt{z} P_2(z), 
\end{equation*}
then the formula $(\ref{polydense})$ implies that 
\begin{align*}\label{polydense1}
   \int_{\Omega} \left | f(z) - P_2( z)-   \frac{P_1(z ) }{\sqrt{z} } \right | ^2 e ^{-\varphi (z)}  d \lambda_z  < \varepsilon.  
\end{align*}
In order to find some polynomial $Q(z)$ such that 
$$
 \int_{\Omega} \left |f(z) - Q(z) \right | ^2 e ^{-\varphi (z)}  d \lambda_z  < 2  \varepsilon.  
$$
It is sufficient to know that 
$$
\int_\Omega \left| \frac{1}{ \sqrt{z}} - R(z)    \right|^2 e^{- \varphi (z)} d\lambda < \varepsilon 
$$
for some polynomial $R(z)$ but this holds by assumption.\\
\end{proof}

We give now an application of Theorem \ref{moonshapeddomain}: Example \ref{moonshapeddomainarbitrarily}. First we need the following lemmas.\\

\begin{lemma} [Riesz Decomposition Theorem, see for example Theorem 3.7.9 in \cite{Ransford}]
Let $u$ be a subharmonic function on a domain $D$ in $\C$, with $u \not\equiv -\infty$. Then,
given a relatively compact open subset $U$ of $D$, we can decompose $u$ as 
$$
u= \int_{ \zeta \in U } \log |z- \zeta | d \mu(\zeta) +  h  
$$
on $U$, where $\mu = \frac{1}{2 \pi} \Delta u |_U$ and $h$ is harmonic on $U$.
\end{lemma}

Let $D,R$ be positive numbers so that $D=\pi R^2.$ Fix $\alpha, 0<\alpha<2.$
\begin{lemma}
Let $z_0\in \mathbb C$ and $A$ be a measurable set in $\mathbb C$ with bounded area  $D.$
Then $\int_A \frac{1}{|z-z_0|^\alpha} d\lambda \leq \frac{R^{2-\alpha}}{2-\alpha}$. Moreover this estimate is sharp.
\end{lemma}
\begin{proof} It suffices to consider the case when $z_0=0.$
The largest integral occurs when the area is a disc centered at $0$. The radius then is given by $\pi R^2=D.$
We get 
$\int_{|z|\leq R} \frac{1}{|z|^\alpha} d \lambda = \int_0^Rt^{1-\alpha}dt= \frac{R^{2-\alpha}}{2-\alpha}.$
\end{proof}
Using the convexity of the exponential function we apply Lemma 2.1 in \cite{Wu-Fornaess2018}
to obtain:

\begin{lemma}\label{estimatefor product}
Suppose $0<\alpha_i, \sum_i \alpha_i=\alpha<2.$ Then
$$\displaystyle{\int_A \Pi_i \frac{1}{|z-z_i|^\alpha_i} d\lambda \leq \frac{R^{2-\alpha}}{2-\alpha}}$$ as in the above Lemma.
\end{lemma}

\begin{corollary}\label{estimatefor phi}
Let $\mu$ be any nonnegative measure with total mass $\alpha, 0<\alpha<2.$ $A$ be a measurable set in $\mathbb C$ with bounded area  $D.$ Then
if $\varphi(z)=\int \log |z-\zeta| d\mu(\zeta)$, we have that $\int_A e^{-\varphi} d\lambda \leq \frac{R^{2-\alpha}}{2-\alpha}.$
\end{corollary}
\begin{proof}
Define $\psi_n(z,\zeta)=\max\{\log|z-\zeta|,-n\} $ and 
$$
\varphi_n(z)=\int\psi_n(z,\zeta)d\mu(\zeta).
$$ 
It suffices to show
that
$$
\int_Ae^{-\varphi_n(z)}d\lambda(z) \leq \frac{R^{2-\alpha}}{2-\alpha}+\frac{1}{n}\;\ \ \ \ \   \forall n\in\mathbb{N}^*.
$$
We fix $n$.
Let $\varepsilon >0.$ We apply Lemma 2.4 \cite{Wu-Fornaess2018}: there exists a finite positive measure $\beta:=\sum_{i=1}^N \alpha_i\delta_{z_i}$ with $\sum_{i=1}^N\alpha_i= \alpha$ such that
$$\int \psi_n(z,\zeta) d\beta(\zeta)\leq \int \psi_n(z,\zeta)d\mu(\zeta) +\varepsilon.
$$
By Lemma  \ref{estimatefor product} we know that
$$
\int_{A} e^{-\sum_i \alpha_i \log|z-z_i|}d\lambda
\leq \frac{R^{2-\alpha}}{2-\alpha}.
$$
So
$$
\int_A e^{-\int \log |z-\zeta|d\beta(\zeta)}d\lambda(z) \leq \frac{R^{2-\alpha}}{2-\alpha}.
$$
Hence
$\displaystyle{\int_Ae^{-\int \psi_n(z,\zeta) d\beta(\zeta)}d\lambda \leq \frac{R^{2-\alpha}}{2-\alpha}}$.
Finally, by choosing $\varepsilon$ small enough we get 
$$
\int_Ae^{-\varphi_n}d\lambda \leq \frac{R^{2-\alpha}}{2-\alpha} + \frac 1 n .
$$
\end{proof}

Define for $n\in\mathbb{N}^*$
\begin{align*}
S_n  &= \left \{\varphi \ \text{is subharmonic on} \ \C:  \ \ \  \text{the mass of
} \ \mu \left(|z|< 1+\frac1 n \right) < 2-\frac 1n,\right.\\ 
&\phantom{=\;\;}\left. \varphi (z) = \psi_n + h_n, \ \ \psi_n = \int _{|\zeta| < 1+ \frac1n} \log|z-\zeta| d\mu \ \text{and} \ |h_n| <n \ \text{on} \  |z|<1 \right \} .
\end{align*}
\begin{lemma}\label{Sn}
Suppose $\varphi$ is subharmonic on $\C$ satisfying condition $(A)$, then for all large enough $n$, $\varphi \in S_n$. 
\end{lemma}
\begin{proof}
Pick $m$ so that the mass of $\mu = \frac{1}{2\pi}\Delta \varphi$ is strictly less than 2 on the disc $\Delta\left(1+\frac1m\right)$. Increasing $m$, we may assume 
\begin{align*}
\mu\left(\vert z\vert< 1+\frac 1m\right)< 2-\frac1m. 
\end{align*}
This remains true for all large $m$, let $\varphi$ be subharmonic on $\C$ satisfying condition $(A)$, write $\varphi = \psi_m + h_m$ and set 
$$
K= \sup\limits_{|z|\leq 1} |h_m (z)|.
$$ 
For $n>m$, we may also write $\varphi = \psi_n + h_n$. We have that 
$$
|\psi _n - \psi _m | \leq 2 \log n \ \ \text{on} \ \ \ |z|<1.
$$
Hence 
 $$
 |h_n| \leq K+ 2 \log n  \ \  \text{on} \ \ \ |z|<1.
 $$
We may choose $n$ so large that 
$$
K+ 2 \log n  \leq n.
$$
\end{proof}
Now we begin to construct Example \ref{moonshapeddomainarbitrarily}. By Lemma \ref{Sn}, we may assume $\varphi \in S_1$. Then, 
$\varphi = \varphi_1  + h_1$ where 
$\varphi_1 = \int_{|\zeta|<2}\log|z-\zeta| d \mu $
and $\mu(|z|<2)< 1$. 
Let $D_1$ be the domain composed of points satisfying the inequalities
\begin{equation*}
|z| < 1,  \ \ \ \      \left | z- \frac{1}{2^2} \right | > \frac34, \ \ \ \ \   \frac {\pi}{ 2^2} < \arg z < 2\pi -\frac {\pi}{ 2^2}.
\end{equation*}
Then $D_1$ is a bounded simply connected domain so that $\mathbb{C} \setminus \overline{D}_1$ is connected.
By Corollary \ref{estimatefor phi} we know that $\int_{D_1}e^{-\varphi_1}d \lambda$ is uniformly bounded. Since $|h_1|<1$, we also get that $\int_{D_1}e^{-\varphi}d \lambda$ is uniformly bounded for all $\varphi \in S_1$. By the Runge theorem there exists a polynomial $P_{n_1}(z)$ with degree $n_1$ such that 
\begin{equation*}
\int_{D_1} \left|  \frac{1}{\sqrt{z}}  - P_{n_1}(z) \right|^2 e^{-\varphi(z)}d\lambda < \frac{1}{2^2} 
\end{equation*}
for all $\varphi \in S_1$.
Choose $0 <\alpha_1< \frac{1}{2^2}$ is sufficiently small so that for the above polynomial $P_{n_1}(z)$ satisfying
\begin{align*}
\int_{\Delta_1} \left|  \frac{1}{\sqrt{z}}  - P_{n_1}(z) \right|^2  e^{-\varphi(z)} d\lambda & \leq  \sup_{\Delta_1} \left| \frac{1}{\sqrt{z}}  - P_{n_1}(z)  \right |^2 \int_{\Delta_1}  e^{-\varphi_1-h_1}d\lambda\\
&\leq  \sup_{\Delta_1} \left| \frac{1}{\sqrt{z}}  - P_{n_1}(z)  \right |^2 \frac{R_1^{2-\alpha}}{2-\alpha}\cdot e\\
& < \frac {1}{ 2^2},
\end{align*}
where $\Delta_1$, of area $\pi R_1^2$ being the points satisfying the inequalities
\begin{equation}
|z| < 1,  \ \ \ \       | z-\alpha_1| > 1- \alpha_1, \ \ \ \ \   | \arg z | \leq \frac {\pi}{ 2^2}.
\end{equation}  
Let $\widetilde{D_2}$ be the domain composed of the points verifying the inequalities
\begin{equation*}
|z| < 1,  \ \ \ \        | z- \alpha_1| > 1-\alpha_1,  \ \ \ \ \   \frac {\pi}{ 2^3} < \arg z < 2\pi -\frac {\pi}{ 2^3}.
\end{equation*}
Set $D_2 = D_1 \cup \widetilde{D_2}$,  this is a Jordan domain. We now consider any $\varphi \in S_2$. Then there exists a polynomial $P_{n_2}(z)$ with degree $n_2$ such that 
\begin{equation*}
\int_{D_2} \left|  \frac{1}{\sqrt{z}}  - P_{n_2}(z) \right|^2 e^{-\varphi(z)}d\lambda < \frac{1}{2^3}. 
\end{equation*}
Choose $0<\alpha_2 <\alpha_1$ so that for the above polynomial $P_{n_2}(z)$ satisfying
\begin{align*}
\int_{\Delta_2} \left|  \frac{1}{\sqrt{z}}  - P_{n_2}(z) \right|^2  e^{-\varphi(z)} d\lambda & \leq  \sup_{\Delta_2} \left| \frac{1}{\sqrt{z}}  - P_{n_2}(z)  \right |^2 \int_{\Delta_2}  e^{-\varphi}d\lambda\\
&\leq  \sup_{\Delta_2} \left| \frac{1}{\sqrt{z}}  - P_{n_2}(z)  \right |^2 \frac{R_2^{2-\alpha}}{2-\alpha}\cdot e^2\\
& < \frac {1}{ 2^3},
\end{align*} 
where $\Delta_2$, of area $\pi R_2^2$ being the points satisfying the inequalities
\begin{equation*}
|z| < 1,  \ \ \ \       | z-\alpha_2| >1- \alpha_2, \ \ \ \ \   | \arg z |\leq \frac {\pi}{ 2^3}.
\end{equation*}
Let $\widetilde{D_3} := \{|z| < 1,  \ \      | z-\alpha_2| >1- \alpha_2,  \ \ \ \ \   \frac {\pi}{ 2^4} < \arg z < 2\pi -\frac {\pi}{ 2^4}\}.$
Set $D_3 = D_2 \cup \widetilde{D_3}$. Proceed as above, we can find a sequence of bounded simply connected domains $D_1, D_2,  \cdots, D_n, \cdots.$ Let $D$ be the limit domain of $D_n$. Then the domain $D$ is bounded by the circle $|z|=1$ and a simple Jordan curve $\Gamma$ tangent to $|z|=1$ in $z=1$. Thus $D$ is a bounded very thin moon-shaped domain and we know that the limit domain $D$ is contained in the limit domain $D_k \cup \Delta_k$, from which it follows that for all $\varphi \in S_k$,
\begin{equation*}
\int_D \left|  \frac{1}{\sqrt{z}}  - P_{n_k}(z) \right|^2 e^{-\varphi_1(z)}d\lambda < \frac{1}{2^{k+1}} +  \frac{1}{2^{k+1}}  = \frac{1}{2^k}.
\end{equation*}
Then polynomials are dense in $H^2(D, e^{-\varphi})$ for any $\varphi$ satisfying condition $(A)$.

\section{Example \ref{moonshapeddomainarbitary} }

First, we prove the following Lemma, already known for Bergman spaces ($\varphi\equiv 0$).

\begin{lemma}\label{bon}
Let $\Omega\Subset \mathbb{C}$ and $\varphi$ a subharmonic function on $\Omega$. Assume that $1\in H^2(\Omega, e^{-\varphi})$ and $p\in\Omega$.
For each $n$, let $f_n\in H^2(\Omega,e^{-\varphi})$ be a function $f_n=a_n(z-p)^n+O((z-p))^{n+1}$, $a_n>0$ maximal with $\Vert f_n\Vert_{\Omega,\varphi}=1$. Then $\lbrace{f_n}\rbrace_{n=0}^{\infty}$ is an orthonormal basis for $H^2(\Omega,e^{-\varphi})$.
\end{lemma}

\begin{proof}
We show that $f_n\perp \lbrace{g\in H^2(\Omega,e^{-\varphi});\; g=O((z-p)^{n+1})}\rbrace$.\\
By contradiction, suppose that there exists such $g\in H^2(\Omega,e^{-\varphi})$ that is not orthogonal to $f_n$. Then for complex-valued $\varepsilon$ small enough, 
\begin{align}
\langle{f_n + \varepsilon g, f_n +\varepsilon g}\rangle & = 1- 2\Re e (\varepsilon\langle{f_n, g}\rangle) + \vert\varepsilon\vert^2\langle{ g, g}\rangle,\\
&= 1- 2\vert \varepsilon\vert \vert \langle{f_n, g}\rangle\vert + \vert\varepsilon g\vert^2,\\
&= t<1.
\end{align}
Hence, $\widetilde{f_n}=\dfrac{f_n +\varepsilon g}{\sqrt{t}}= \dfrac{a_n}{\sqrt{t}}(z-p)^n + O_t((z-p))^{n+1})$ contradicts the maximality of the coefficients.\\

Suppose now that $g\perp \lbrace{f_n}\rbrace$ such that $g=b_n(z-p)^n + O((z-p))^{n+1}$ with $b_n=a_n$.
We get $g-f_n=O((z-p))^{n+1}$ that is orthogonal to $f_n$ i.e 
\begin{align*}
\langle{g-f_n, f_n}\rangle = \langle{g, f_n}\rangle -  \langle{f_n, f_n}\rangle = 0.
\end{align*}
but
$\langle{g, f_n}\rangle =0 \neq 1=\langle{f_n,f_n}\rangle$. 
\end{proof}
Let $M$ be a moon-shaped domain with multiple boundary point $Q$, and $p$ be an interior point of $M$.
For each $n\geq 2$, denote $M_n=M\setminus B(Q,1/n)$ where $B(Q,1/n)$ is the disc of center $Q$ and radius $1/n$. 
Let $\varphi$ be a non-negative subharmonic function on $\mathbb{C}$ and $f\in H^2(M_n,e^{-\varphi})$. By Corollary \ref{coro:Jordandense}, for every $\varepsilon>0$, there exists a polynomial $P$ such that $\Vert f -P\Vert_{M_n, \varphi}<\varepsilon$. 

\begin{lemma}\label{lemma: approxcplmt}
Under the previous assumptions, there exists a subharmonic function $\widetilde{\varphi}$ on $\mathbb{C}$ such that $\widetilde{\varphi}=\varphi$ on $M_n$ and 
$$\Vert P\Vert_{B(Q,1/n)\cap M, \widetilde{\varphi}}<\varepsilon.$$
\end{lemma}

\begin{proof}
Because $\overline{M}_n$ is polynomially convex, there exists a subharmonic function $\gamma$ on $\mathbb{C}$ such that $\gamma =0$ on $M_n$ and $\gamma>0$ outside: To find $\gamma$, note that for every point $q$ outside $\overline{M}_n$ there exists a polynomial $P_q$  such that $|P_q(q)|>1$ and $|P_q|<1$ on $\overline{M}_n.$ We choose a convex function $\chi(x)$ which vanishes when $x\leq 1$ and is strictly positive
when $x>1.$ Then $\chi\circ |P_q|^2$ is subharmonic and vanishes on $M_n$ while
it is strictly positive in a neighborhood of $q.$ Then one can  define $\gamma=\sum_m \varepsilon_m
\chi\circ |P_{q_m}|^2$ for suitable choices. By choosing $\widetilde{\varphi}=\varphi + L\gamma$ for $L$ large enough, we get 
$$\Vert P\Vert_{B(Q,1/n)\cap M,\widetilde{\varphi}}\leq \Vert P\Vert_{B(Q,1/n)\cap M,L\gamma}.$$
By taking the limit as $L$ tends to $+\infty$, the second term of the previous estimate tends to $0$.
\end{proof}

By Lemma \ref{lemma: approxcplmt}, we can construct inductively an increasing sequence of non-negative subharmonic functions $\varphi_n$ on $M$ and by Lemma \ref{bon}, we can find $f_0^n,\dots,f_n^n\in H^2(M_n,e^{-\varphi_n})$ such that $f_j^n=a_j^n(z-p)^{j} + O((z-p)^{j+1})$ with $a_j^n>0$ maximal and $\Vert f_j^n\Vert_{M_n,\varphi_n}=1$.
Then, there exist polynomials $P_j^n$ and $\varphi_{n+1}$ large enough by Lemma \ref{lemma: approxcplmt} such that 
\begin{equation}\label{fj-Pj: Mn}
\Vert f_j^n-P_j^n\Vert_{M_n,\varphi_n}<\dfrac{1}{n}
\end{equation}
and
\begin{equation}\label{P_j:cplmt}
\Vert P_j^n\Vert_{B(Q,1/n)\cap M,\varphi_{n+1}}<\dfrac{1}{n}
\end{equation}
Moreover, by Remark \ref{rmk:Taylor}, we can choose 
$$P_j^n(z) = a_j^n(z-p)^j+O((z-p)^{j+1}).$$
Let $\varphi:=\lim_{n\to+\infty} \varphi_n$. 
Hence $\Vert f_j^n\Vert_{M_n,\varphi}=1$ and \eqref{fj-Pj: Mn} and \eqref{P_j:cplmt} hold with respect to $\varphi$.\\

\begin{lemma}\label{lemma: Pjnorm}
Under the previous assumptions, the polynomials $P_j^n$, $j=0,\dots, n$ built inductively as previously verify the following property
$$\Vert P_j^n\Vert_{M,\varphi}\leq 1+\dfrac{2}{n}.$$
\end{lemma}

\begin{proof}
Combining the properties of $f_j^n$ seen in Lemma \ref{bon} and Lemma \ref{lemma: approxcplmt}
\begin{align*}
\Vert P_j^n\Vert_{M,\varphi}& \leq \Vert P_j^n\Vert_{M_n,\varphi} + \Vert P_j^n\Vert_{B(Q,1/n)\cap M,\varphi},\\
&\leq \Vert P_j^n\Vert_{M_n,\varphi_n} + \Vert P_j^n\Vert_{B(Q,1/n)\cap M,\varphi_{n+1}},\\
&\leq \Vert P_j^n-f_j^n\Vert_{M_n,\varphi_n} + \Vert f_j^n\Vert_{M_n,\varphi_n} + \dfrac{1}{n},\\
&\leq 1 + \dfrac{2}{n}.
\end{align*}
\end{proof}

We are now able to give the details of the construction of Example \ref{moonshapeddomainarbitary}. 

By Lemma \ref{lemma: Pjnorm}, $P_j^n$ converges weakly to $P_j$ in $H^2(M,e^{-\varphi})$ such that 
$$\Vert P_j\Vert_{M,\varphi}\leq 1$$
and 
$$P_j=(\lim_{n\to +\infty} a_j^n)(z-p)^j+ O((z-p)^{j+1}).
$$
In particular, the limit $\lim_{n\to +\infty} a_j^n$ gives optimal coefficients for $P_j$. It follows that $\Vert P_j\Vert_{M,\varphi}=1$ and $P_j$ is an orthonormal basis for $H^2(M,e^{-\varphi})$ by Lemma \ref{bon}.

Let $f\in H^2(M,e^{-\varphi})$ and $\varepsilon>0$. We can express $f$ as $\displaystyle{f=\sum_{j=0}^{\infty} A_jP_j}$.

For $N$ large enough, 
\begin{align*}
\Vert \sum_{j=0}^N A_jP_j - f\Vert_{M,\varphi}<\dfrac{\varepsilon}{2}.
\end{align*}
For $n$ large enough, we get finally
\begin{align*}
\Vert \sum_{j=0}^N A_jP_j^n - f\Vert_{M,\varphi}<{\varepsilon}.
\end{align*}

\section{Proof of Theorem \ref{moonshapeddomainnot}}

\begin{proof} [Proof of Theorem \ref{moonshapeddomainnot}]
The proof relies on \cite{Gaier80}, Chapter I, section 3. We reason by contradiction.\\ 
\indent Let $\Gamma$ be a circle as in Figure 1. Then there exists a positive constant $C$ such that
$d_{\partial \Omega}(z)\geq C|z-1|^2$ for all $z\in \Gamma.$ 
Choose a function $f\in {H}^2(\Omega, e^{-\varphi})$ which does not extend holomorphically to the inside of the small circle. (One can first choose any nontrivial holomorphic function in the ${H}^2(\Omega, e^{-\varphi})$.
Pick a suitable point $p$ inside the inner circle and observe that $\frac{f}{z-p}$ is still in the $H^2$ space.)\\
\indent Assume that there exists a sequence of polynomials $P_n$ so that $\Vert f-P_n\Vert_{\Omega,\varphi}$ tends to $0$ as $n$ tends to $+\infty$. We obtain a contradiction by showing that $f$ extends analytically in the interior of $\Gamma$ so in the hole (see Figure 1), which is impossible.\\
\begin{figure}
\begin{center}
\begin{tikzpicture}
\fill[pattern=north west lines] (0, 0) circle (2);
\fill[white] (1.3,0) circle (0.7);
\draw[->] (-3,0)-- (3,0);
\draw (3,0) node [right] {$x$};
\draw[->] (0,-3)-- (0,3);
\draw (0,3) node [left] {$y$};
\draw (-0.3,-0.3) node {$0$};
\draw (2,0.3) node [right] {$z=1$};
\draw (-1.5, -1.5) node [left] {$G$};
\draw (2,0) node {$\bullet$} ;
\draw (2.2,0) node [below] {$Q$} ;
\draw[red] (0.5, 0) circle (1.5);
\draw (-1, 1) [left] node {$\color{red}{\Gamma}$};
\end{tikzpicture}
\end{center}
\caption{A general moon-shaped domain including the curve $\Gamma$ passing through the multiple boundary point $Q$.}
\end{figure}
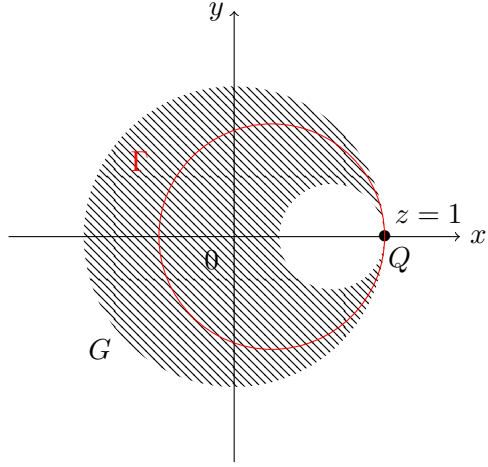
For each $w\in \Gamma\setminus \lbrace{Q}\rbrace$, by the mean value property for subharmonic functions, we get for any $n,m\in\mathbb{N}$, 
\begin{align}
\vert P_n(w)-P_m(w)\vert^2&\leq \dfrac{1}{\pi d^2_{\partial\Omega}(w)}\int_{B(w, d_{\partial\Omega}(w))} \vert P_n(z)-P_m(z)\vert^2 d\lambda,\notag\\
&\leq \dfrac{\widetilde{C}}{\pi d^2_{\partial\Omega}(w)}\int_{B(w, d_{\partial\Omega}(w))} \vert P_n(z)-P_m(z)\vert^2 e^{-\varphi(z)} d\lambda,\label{Cphi}\\
&\leq \dfrac{\widetilde{C}}{\pi d^2_{\partial\Omega}(w)}\Vert P_n(z)-P_m(z)\Vert^2_{\Omega,\varphi},
\label{eqPnPm}
\end{align}
where we have used in \eqref{Cphi} the fact that $\varphi$ is bounded above in $\Omega$ with $\widetilde{C}$ a positive constant.
Because $d_{\partial\Omega}(z)\geq C\vert z-1\vert^2$ by assumption, we obtain from \eqref{eqPnPm} and for each $w\in \Gamma\setminus \lbrace{Q}\rbrace$,
\begin{align}
\vert (w-1)^2(P_n(w)-P_m(w))\vert &\leq \dfrac{d_{\partial\Omega}(w)}{C}\dfrac{\sqrt{\widetilde{C}}}{\sqrt{\pi} d_{\partial\Omega}(w)}\Vert P_n(z)-P_m(z)\Vert_{\Omega,\varphi},\notag\\
&\leq C'\Vert P_n(z)-P_m(z)\Vert_{\Omega,\varphi}\label{powerp}
\end{align}
where $C'>0$. The inequality \eqref{powerp} holds for each point $w\in \Gamma$ so in the interior of $\Gamma$. Hence, the sequence $((z-1)^2P_n)$ converges uniformly in the interior of $\Gamma$ to $(z-1)^2 f$. 

\end{proof}

\textbf{Acknowledgements} The first author was supported by Rannis-grant 152572-051. The second author and the third author were supported in part by the Norwegian Research Council grant number 240569, the third author was also supported by NSFC grant 11601120.  The authors give thanks to Dr. Zhonghua Wang for his valuable comments in the proof of theorem \ref{counterexinC}
and the referees for their valuable suggestions.

\end{document}